\documentclass[11pt]{article}
\usepackage{amsmath,amsfonts,amstext,amsthm,epsfig}
\begin{document}
\font\tenrm=cmr10

%
 \newtheorem{thm}{Theorem}[section]
 \newtheorem{cor}[thm]{Corollary}
 \newtheorem{lem}[thm]{Lemma}
 \newtheorem{prop}[thm]{Proposition}
 \theoremstyle{definition}
 \newtheorem{defn}[thm]{Definition}
 \theoremstyle{remark}
 \newtheorem{rem}[thm]{Remark}
 \newtheorem*{ex}{Example}

\renewcommand{\theequation}{\thesection.\arabic{equation}}

\title{\bf A new result on backward uniqueness for parabolic operators}

\author{ {\sc Daniele Del Santo} {\small and} {\sc Martino Prizzi}\\[0.2 cm] {\small Dipartimento di
Matematica e Informatica, Universit\`a di Trieste}\\[-0.1 cm]
{\small Via A.~Valerio 12/1, 34127 Trieste, Italy} }

\date{\today}
\maketitle

\begin{abstract}
Using Bony's paramultiplication we improve a result obtained in \cite{DSP} for operators having coefficients non-Lipschitz-continuous with respect to $t$  
but  ${\mathcal C}^2$  with respect to  $x$, showing that the same result is valid when ${\mathcal C}^2$ regularity is replaced by Lipschitz regularity in $x$.
\end{abstract}




\section{Introduction}

In this note we consider the following backward parabolic operator
\begin{equation} L=\partial_t+ \sum_{i, j = 1}^n \partial_{x_j} (a_{jk} (t,x)
 \partial_{x_k} )+\sum_{ j = 1}^n b_j(t,x)\partial_{x_j}+c(t,x).\label{k1}
\end{equation} We assume that all coefficients are  defined in $[0,T]
\times {\mathbb R}^n_x$,
measurable  and bounded; $( a_{jk}(t,x))_{jk}$ is a real symmetric matrix for
all $(t,x)\in [0,T]\times
{\mathbb R}^n_x$ and there exists $\lambda_0\in (0,1]$ such that
$$
\sum_{j, k = 1}^n  a_{jk} (t,x)
\xi_j \xi_k\geq \lambda_0|\xi|^2\label{k2}
$$ for all $(t,x)\in [0,T]\times  {\mathbb R}^n_x$ and $\xi\in {\mathbb R}^n_\xi$.

Given a functional space ${\cal H}$ (in which it makes sense to look for
the solutions of the
equation $Lu=0$) we say that the operator $L$ has the ${\cal
H}$--uniqueness property if, whenever $u\in
{\cal H}$, $Lu=0$ in $[0,T]\times  {\mathbb R}^n_x$ and
$u(0,x)=0$ in ${\mathbb R}^n_x$, then $u=0$ in $[0,T]\times  {\mathbb R}^n_x$.

We choose ${\mathcal H}$ to be the space of functions 
\begin{equation}
 {\mathcal H}= H^1((0,T), L^2({\mathbb R}^n_x))\cap L^2( (0,T), H^2({\mathbb R}^n_x)).
\end{equation} 
This choice is natural, since 
it follows from elliptic regularity results
(see e.g. Theorem 8.8 in \cite{GT}) that the domain of the operator $-\sum_{j,k = 1}^n \partial_{x_j} (a_{jk} (t,x)
\partial_{x_k} )$ in $L^2(\mathbb R^n)$ is $H^2(\mathbb R^n)$ for all $t\in[0,T]$.

The problem we are interested in is the following:  find the minimal
regularity on the coefficients
$a_{jk}$ ensuring the ${\cal H}$--uniqueness property to $L$.

A classical result of  Lions and Malgrange \cite {LM} (see for related or
more general results
\cite {Mil}, \cite{BT}, \cite{Gh}) shows that a sufficient condition for backward uniqueness 
is given by the assumption that the map $t\mapsto a_{jk}(t,\cdot)$ be Lipschitz continuous
from $[0,T]$ to $L^\infty(\mathbb R^n)$.

On the other hand the well
known example of Miller
\cite{Mil} (where an operator, having coefficients which are
H\"older--continuous of order $1/6$
with respect to $t$ and $C^\infty$ with respect to $x$, does not have the
uniqueness property) shows
that a certain amount of regularity on the $a_{jk}$'s with respect to $t$ is necessary for the
$ {\cal H}$--uniqueness.

In our previous paper \cite{DSP}, we proved
 the $ {\cal
H}$--uniqueness property for
the operator (\ref{k1}) when the coefficients $a_{jk}$ are $C^2$ in the $x$
variables and
non--Lipschitz--continuous in $t$. The regularity in $t$ was given in
terms of a modulus of
continuity $\mu$ satisfying the so called Osgood condition
$$
\int_0^1{1\over \mu(s)}\, ds=+\infty.
$$

This uniqueness result was a consequence of a Carleman estimate in which the
weight function depended
on the modulus of continuity;  such kind of weight functions in Carleman
estimates were
introduced by Tarama \cite {T} in the case of second order elliptic
operators. In obtaining our
Carleman estimate, the integrations by parts, which couldn't be used since the
coefficients were not
Lipschitz--continuous,  was replaced by a microlocal approximation
procedure.


In \cite{DSP}
a technical difficulty in the estimate of a commutator led to imposing on $a$ the the $C^\mu$ regularity with respect to $t$,  together with the 
 ${\mathcal C}^{2}$  regularity with respect to $x$.
In \cite{DS}  this statement was improved, as it was shown that under  the Osgood condition for $\mu$, the $C^\mu$ regularity with respect to $t$,  together with the 
H\"older  ${\mathcal C}^{1,\varepsilon}$  regularity with respect to $x$, is sufficient for the same uniqueness result. 
The proof followed the same pattern as the one in \cite{DSP}, the only difference being in the introduction of a paradifferential operator (actually a simple paramultiplication) in place of the second order part of the operator $ L$. In the present paper, we further improve the result of \cite{DS}, showing that ${\mathcal C}^{1,\varepsilon}$ regularity can be replaced by Lipschitz regularity in $x$.
In order to achieve our result, we introduce a modified paramultiplication. We obtain a Carleman estimate in a space $H^{-s}$, with $0<s<1$, instead of the classical estimate in $L^2$. However, with such estimate we can repeat the arguments of \cite{DSP} and regain the desired uniqueness property.
The estimate of the commutator, in \cite{DS} and in the present  case, is made more effective by a theorem due to Coifman and Meyer \cite[Th. 35]{CMey} (see also, for a similar use of that theorem, \cite[Prop. 3.7]{CMet}).


\section{Definitions and result}

\begin{defn}
A function $\mu$ is said to be a {\it modulus of continuity} if $\mu$ is continuous, concave and strictly increasing on $[ 0,1]$, with $\mu(0)=0$ and $\mu(1)=1$.
Let $I \subseteq {\mathbb R}$ and
let
$\varphi \,:\, I\to {\mathcal B}$, where ${\mathcal B}$ is a Banach space. We say
that $\varphi\in
C^\mu(I,{\mathcal B})$ if $\varphi\in L^\infty(I,{\mathcal B})$ and
$$
\sup_{0<|t-s|<1\atop t,s\in I}{\|\varphi(t)-\varphi (s)\|_{\mathcal
B}\over\mu(|t-s|)}<+\infty.
$$	\label{defmod}
\end{defn}
It is immediate to verify the following properties
\begin{itemize}
\item $\mu(s)\geq s$ for all $s\in [0,1]$;
\item the function $s\mapsto {\mu(s)/s}$ is decreasing on  $\;]0, 1]$;
\item there exists $\lim _{s\to 0^+} {\mu(s)/ s}$;
\item the function $\sigma\mapsto\sigma \mu(1/\sigma)$ is 
increasing on $[1,+\infty[$;
\item the function $\sigma\mapsto 1/(\sigma^2\mu(1/\sigma))$ is 
decreasing on on $[1,+\infty[$.
\end{itemize}

\begin{defn} A modulus of continuity is said to satisfy the {\it Osgood condition} if 
\begin{equation}
\int_0^1{1\over \mu(s)}\, ds=+\infty. \label{osgood}
\end{equation}
\label{defos}
\end{defn}

\begin{thm}
Let $L$ be the operator
\begin{equation} L=\partial_t+ \sum_{j,k = 1}^n \partial_{x_j} (a_{jk} (t,x)
\partial_{x_k} )+\sum_{ j = 1}^n b_j(t,x)\partial_{x_j}+c(t,x),\label{opL}
\end{equation} 
where all the coefficients are supposed to be defined in $[0,T]
\times {\mathbb R}^n_x$, measurable and bounded; let the coefficients
$b_j$ and $c$ be complex valued; let $( a_{jk}(t,x))_{jk}$ be a real
symmetric matrix for all $(t,x)\in [0,T]\times {\mathbb R}^n_x$
and suppose that there exists $\lambda_0\in (0,1)$ such that
\begin{equation}
 \sum_{j, k = 1}^n a_{jk} (t,x)
\xi_j \xi_k\geq \lambda_0|\xi|^2,\label{ell}
\end{equation} 
for all $(t,x)\in [0,T]\times {\mathbb R}^n_x$ and $\xi\in
{\mathbb R}^n_\xi$.
Let ${\mathcal H}$ be the space of functions 
\begin{equation}
 {\mathcal H}= H^1((0,T), L^2({\mathbb R}^n_x))\cap L^2( (0,T), H^2({\mathbb R}^n_x)).\label{H}
\end{equation} 
Let $\mu$ be a modulus of continuity satisfying the Osgood condition.
Suppose that 
\begin{equation}
a_{jk}\in  C^\mu([0,T],L^\infty({\mathbb
R}^n_x))\cap C([0,T],{Lip\,}({\mathbb R}^n_x)),\label{reg}
\end{equation} 
for all $j,k=1\dots,n$.

Then $L$ has the ${\mathcal H}$--uniqueness
property, i.e. if $u\in {\mathcal H}$, $Lu=0$ in $[0,T]\times {\mathbb R}^n_x$ and
$u(0,x)=0$ in ${\mathbb R}^n_x$, then $u=0$ in $[0,T]\times
{\mathbb R}^n_x$.
\label{t1}
\end{thm}

\noindent{\bf Remark.} The choice of the space $\mathcal H$ is natural, since 
it follows from elliptic regularity results
(see e.g. Theorem 8.8 in \cite{GT}) that the domain of the operator $-\sum_{j,k = 1}^n \partial_{x_j} (a_{jk} (t,x)
\partial_{x_k} )$ in $L^2(\mathbb R^n)$ is $H^2(\mathbb R^n)$ for all $t\in[0,T]$.


\section{Proof}


\subsection{Modulus of continuity and Carleman estimate}
Theorem \ref{t1} will follow from a Carleman
estimate in Sobolev spaces with negative index. The weight function in the Carleman estimate will be obtained from the modulus of continuity. The crucial idea 
of linking the weight function to the regularity of the coefficients goes back to the paper \cite{T} in which a uniqueness result for elliptic operators with non-Lipschitz-continuous coefficients was proved.  

We define 
$$ \phi(t)=\int_{ 1\over
t}^1{1\over \mu(s)}\; ds. 
$$ 
The function $\phi$ is a strictly
increasing $C^1$ function. From (\ref{osgood}) we have
$\phi([1,+\infty[)=[0,+\infty[$; moreover $\phi'(t)= 1/(t^2\mu(1/t))>0$
for all $t\in [1,+\infty[$. We set 
$$
\Phi(\tau)=\int_0^\tau\phi^{-1}(s)\;ds. 
$$ 
We obtain
$\Phi'(\tau)=\phi^{-1}(\tau)$ and consequently $\lim_{\tau\to
+\infty}\Phi'(\tau)=+\infty$. Moreover
\begin{equation}
\Phi''(\tau)=(\Phi'(\tau))^2\mu({1\over \Phi'(\tau)})
\label{diffeq}
\end{equation} 
for all $\tau\in [0,+\infty[$ and, as the function
$\sigma\mapsto \sigma\mu(1/\sigma)$ 
is increasing on $[1,+\infty[$, we deduce that
\begin{equation}
\lim_{\tau\to +\infty}\Phi''(\tau)=\lim_{\tau\to
+\infty}(\Phi'(\tau))^2\mu({1\over
\Phi'(\tau)})=+\infty.
\label{weightatinfinity}
\end{equation} 
Now we state the Carleman estimate.

\begin{prop} 
For all $s\in (0,1)$, there  exist $\gamma_0$, $C>0$  such that
\begin{equation}
\begin{array}{ll}
\displaystyle{
\int_0^{T\over 2} e^{{2\over
\gamma}\Phi(\gamma(T-t))}\|\partial_t u+ \sum_{j,k = 1}^n \partial_{x_j} (a_{jk} (t,x)
\partial_{x_k} u) \|^2_{H^{-s}}\;
dt}\\[0.5 cm]
\qquad\qquad\qquad\displaystyle{\geq C\gamma^{1\over 2}\int_0^{T\over 2} e^{{2\over
\gamma}\Phi(\gamma(T-t))}
(\|\nabla_x u\|^2_{H^{-s}}+\gamma^{1\over
2}\|u\|^2_{H^{-s}})\; dt,}
\end{array}
\label{Carlest}
\end{equation} for all $\gamma>\gamma_0$ and for all $u\in
C^\infty_0({\mathbb R}^{n+1})$
such that $\,{\rm supp}\, u\subseteq [0, T/2]\times {\mathbb R}^n_x$  (the symbol $\nabla_x f$ denotes the gradient of $f$ with respect to the $x$ variables).
\label{propcarlest}
\end{prop}
The way of obtaining the ${\mathcal H}$-uniqueness from the inequality (\ref{Carlest}) is a standard procedure, the details of which, in the case of a Carleman estimate in $L^2$, can be found in \cite[Par. 3.4]{DSP}. 



\subsection{Paraproducts}


\subsubsection{Littlewood-Paley decomposition}

We review some known results on Littlewood-Paley decomposition and related topics. More can be found in \cite{B},  \cite[Ch. 4 and Ch. 5]{M} and \cite[\hbox{Par. 3}]{CMet}.

Let $\chi\in C^\infty_0({\mathbb R})$, $0\leq \chi\leq 1$, even and such that
$\chi(s)=1$ for $|s|\leq 11/10$ and $\chi(s)=0$ for $|s|\geq 19/10$. For $k\in{\mathbb Z}$ and $\xi\in{\mathbb R}^n$, let us consider $\chi_k(\xi)=\chi(2^{-k}|\xi|)$, let's denote $\tilde \chi_k(x)$ its inverse Fourier transform and let's define the operators
$$
\begin{array}{c}
\displaystyle{S_{-1}u=0,\quad{\rm\  and}\ \  S_k u= \tilde\chi_k * u= \chi_k(D_x)u,}\\[0.3 cm]
\displaystyle{\Delta_0 u=S_0u,\quad{\rm\  and,\  for}\ \  k\geq 1, \ \ \Delta_ku=S_ku-S_{k-1}u.}
\end{array}
$$
In the following propositions we recall the characterization of Sobolev spaces and Lipschitz-continuos functions via Littlewood-Paley decomposition (see \cite[Prop. 4.1.11]{M}, \cite[Prop. 3.1 and Prop. 3.2]{CMet}  and \cite[Lemma 3.2]{GR}.

\begin{prop}
Let $s\in {\mathbb R}$. A temperate distribution $u$ is in $H^s$ if and only if the following two conditions hold
\begin{itemize}
\item [{i)}] for all $k\geq 0$, $\Delta_ku\in L^2$;
\item [ii)] the sequence $(\delta_k)_k$, with $\delta_k= 2^{ks}\|\Delta_k u\|_{L^2}$, is in $l^2$.
\end{itemize}

Moreover there exists $C_s\geq 1$ such that, for all $u\in H^s$,
$$
{1\over C_s}\|u\|_{H^s}\leq \big(\sum_{k=0}^{+\infty} \delta_k^2\big)^{1/2} \leq C_s\|u\|_{H^s}.
$$
\label{prop car sob}
\end{prop}
\begin{prop}
Let $s\in {\mathbb R}$ and $R>0$. Let $(u_k)_k$ a sequence of functions in $L^2$
such that
\begin{itemize}
\item [{i)}] the support of the Fourier transform of $u_0$ is contained in $\{|\xi|\leq R\}$ and the support of  Fourier transform  of $u_k$ is contained in $\{{1\over R} \,2^k\leq |\xi|\leq R\, 2^k\}$, for all $k\geq 1$;
\item [ii)] the sequence $(\delta_k)_k$, with $\delta_k= 2^{ks}\|u_k\|_{L^2}$, is in $l^2$.
\end{itemize}
Then the series $\sum_k u_k $ is converging, with sum $u$,  in $H^s$ and the norm of $u$ in $H^s$ is equivalent to the norm of $(\delta_k)_k$ in $l^2$. 

When $s>0$ it is sufficient to assume the 
Fourier transform of $u_k$ to be contained in $\{|\xi|\leq R 2^k\}$, for all $k\geq 1$.
\label{prop car sob 2}
\end{prop}
\begin{prop}
 A bounded function $a$ is in $Lip$, the space of bounded Lipschitz continuous functions defined on ${\mathbb R}^N_x$,   if and only if  
 $$
 \sup_{k\in {\mathbb N}}\|\nabla_x(S_ka)\|_{L^\infty}<+\infty.
 $$

Moreover there exists $C>0$ such that if $a\in Lip$, then 
$$
\|\Delta_k a\|_{L^\infty} \leq C\|a\|_{Lip}\, 2^{-k} \quad{\it and}\quad
\|\nabla_x(S_ka)\|_{L^\infty}\leq   C\|a\|_{Lip}
$$ 
(where $\|f\|_{Lip}=\|f\|_{L^\infty}+\|\nabla_x f\|_{L^\infty}$).
\label{prop car lip}
\end{prop}


\subsubsection{Bony's modified paraproduct}

Let $a\in L^\infty$. The  Bony's paraproduct  of $a$ and $u\in H^s$ (see \cite[Par. 2]{B}) is defined as
$$
T_au=\sum_{k=3}^{+\infty} S_{k-3} a\Delta_k u.
$$
We modify the definition of paraproduct introducing the following operator
$$
T^m_au=S_{m-1}aS_{m+1}u +\sum_{k=m+2}^{+\infty} S_{k-3} a\Delta_k u.
$$
where $m\in {\mathbb N}$ (remark that $T_a=T_a^0$).
Useful properties of the  (modified) paraproduct are contained in the following propositions (see also \cite[Prop. 5.2.1]{M}, \cite[Prop. 3.4]{CMet}).
\begin{prop}
Let $m\in {\mathbb N}$,  $s\in {\mathbb R}$ and $a\in L^\infty$. 

Then $T^m_a$ maps $H^s$ into $H^s$ and
\begin{equation}
\label{estT_a}
\|T^m_au\|_{H^s}\leq C_{m, s} \|a\|_{L^\infty} \|u\|_{H^s}.
\end{equation}

Let $m\in {\mathbb N}$, $s\in (0,1)$ and $a\in Lip$. 

Then $u \mapsto au  -T^m_a u$ maps $H^{-s}$ into $H^{1-s}$  and 
\begin{equation}
\label{a-Ta}
\|au-T^m_au\|_{H^{1-s}}\leq C_{ m, s} \|a\|_{Lip} \|u\|_{H^{-s}}.
\end{equation}

\label {prop paraprod}
\end{prop}

\begin {proof}
We prove only the second part of the statement. We have
$$
au-T^m_au= \sum_{k=\max\{3, m\}}^{+\infty}\Delta_k a S_{k-3}u +\sum_{k=m}^{+\infty} (\sum_{{j\geq 0 \atop |j-k|\leq 2}}\Delta_k a\Delta_j u).
$$
We remark that the support of the Fourier transform of $\Delta_k a S_{k-3}u$ is contained in $\{2^{k-2} \leq |\xi|\leq 2^{k+2}\}$. Moreover, by Proposition \ref{prop car lip},
we have that
$$
\|\Delta_k a S_{k-3}u\|_{L^2} \leq \|\Delta_k a\|_{L^\infty}\| S_{k-3}u\|_{L^2}
\leq C   \| a\|_{Lip}\, 2^{-k}\, \sum_{j=0}^{k-3}2^{js}\delta_j
$$
where $\delta_j=2^{-js}\|\Delta_ju\|_{L^2}$. From Proposition \ref{prop car sob} we have  that $(\delta_j)_j\in l^2$ and its $l^2$ norm is equivalent to the $H^{-s}$ norm of $u$. On the other hand, setting
$$
\tilde \delta_k= \sum_{j=0}^{k}2^{(j-k)s}\delta_j,
$$
we have that $(\tilde \delta_k)_k\in l^2$ and $\|(\tilde \delta_k)_k\|_{l^2}\leq C_s 
\|( \delta_k)_k\|_{l^2}$. Consequently
$$
\|\Delta_k a S_{k-3}u\|_{L^2} \leq C_s  \| a\|_{Lip} 2^{-k(1-s)}\tilde \delta_k,
$$
and then, by Proposition \ref{prop car sob 2} we have that $\sum_{k=\max\{3, m \}}^{+\infty}\Delta_k a S_{k-3}u \in H^{1-s}$ with 
\begin{equation}
\|\sum_{k=\max\{3, m\}}^{+\infty}\Delta_k a S_{k-3}u \|_{H^{1-s}}\leq C_{m, s}  \| a\|_{Lip} \|u\|_{H^{1-s}}.
\label {a}
\end{equation}

Next, we see that, for $k\geq 2$,
$$
\sum_{k=m}^{+\infty} (\sum_{{j\geq 0 \atop |j-k|\leq 2}}\Delta_k a\Delta_j u)
= \sum_{k=m}^{+\infty} \Delta_{k} a \Delta_{k-2} u +\dots + \sum_{k=m}^{+\infty} \Delta_{k} a \Delta_{k+2} u,
$$
with a slight modification in the case $k=0,\ 1$.
We have that the support of the Fourier transform of $\Delta_k a \Delta_{k-2} u$ is contained in $\{|\xi|\leq 2^{k+2}\}$ and similarly for the other four terms, e.g. the support of the Fourier transform of $\Delta_k a \Delta_{k+2} u$ is contained in $\{|\xi|\leq 2^{k+4}\}$.
Moreover
$$
\|\Delta_k a  \Delta_{k-2} u\|_{L^2}\leq \|\Delta_k a\|_{L^\infty}\|  \Delta_{k-2} u\|_{L^2}
\leq C_s \| a\|_{Lip}\,2^{-k(1-s)} \delta_{k-2}.
$$
Again from Proposition \ref{prop car sob 2} we have that 
$\sum_{k=m}^{+\infty} \Delta_k a \Delta_{k-2} u\in H^{1-s}$ and
$$
\|\sum_{k=m}^{+\infty} \Delta_{k-2} a \Delta_k u\|_{H^{1-s}}\leq C_{m, s} \| a\|_{Lip}\, \|u\|_{H^{-s}}.
$$
Arguing similarly for the other terms we have that $\sum_{k=m}^{+\infty} (\sum_{{j\geq 0 \atop |j-k|\leq 2}}\Delta_k a\Delta_j u)\in  H^{1-s}$ and
\begin{equation}
 \|\sum_{k=m}^{+\infty} (\sum_{{j\geq 0 \atop |j-k|\leq 2}}\Delta_k a\Delta_j u)\|_{
 H^{1-s}}\leq C_{m.s} \| a\|_{Lip} \,\|u\|_{H^{-s}}.
 \label{b}
\end{equation}
The conclusion of the proof of the proposition is reached  putting together (\ref{a}) and (\ref{b}).
\end{proof}

As pointed out in \cite[Par. 3]{CMet}, the positivity of the function $a$ does not imply, for all $m\geq 0$, the positivity of $T^m_a$. Nevertheless the following proposition holds (see \cite[Cor. 3.12]{CMet}).

\begin{prop}
Let $a\in Lip$
and suppose that $a(x)\geq \lambda_0>0$ for all $x\in{\mathbb R}^n$. Then there exists $m$ depending on $\lambda_0$ and $\|a\|
_{Lip}$ such that
\begin{equation}
\label{pos}
{\rm Re}\big(T^m_au, u\big)_{L^2}\geq {\lambda_0\over 2}\|u\|_{L^2},
\end{equation}
for all $u\in L^2$ (here $\big(\cdot, \cdot)_{L^2}$ denotes the scalar product in $L^2$).
A similar result is valid for vector valued functions when  $a$ is replaced by a  positive symmetric matrix $(a_{jk})_{j,k}$.  
\label {prop pos}
\end{prop}

We state now a property of commutation which will be crucial 
in the proof of the Carleman estimate (see \cite[Prop.3.7]{CMet}). 

\begin{prop}
Let $m\in {\mathbb N}$, $a\in Lip$, $s\in (0,1)$ and $u\in H^{1-s}$. 

Then 
\begin{equation}
\big(\sum_{\nu=0}^{+\infty} 2^{-2\nu s}\|\partial_{x_j}([\Delta_\nu, T^m_a]\partial_{x_h} u)\|^2_{L^2}\big)^{1/2}\leq C_{m, s} \|a\|_{Lip}\, \|u\|_{H^{1-s}}
\label{estcomm}
\end{equation}
(where $[A, B]$ denotes the commutator between the operators $A$ and $B$, i.e. $[A, B]w= A(Bw)-B(Aw)$).
\label {prop comm}
\end{prop}

\begin{proof}
We start remarking that 
$$
[\Delta_\nu, T^m_a]w=[\Delta_\nu, S_{m-1}a]S_{m+1}w+ \sum_{k=m+2}^{+\infty}[\Delta_\nu, S_{k-3}a]\Delta_k w,
$$
and consequently
$$
\begin{array}{ll}
\displaystyle{\partial_{x_j}([\Delta_\nu, T^m_a]\partial_{x_h} u) =
\partial_{x_j}([\Delta_\nu, S_{m-1}a]S_{m+1}(\partial_{x_h} u))}
\\ \qquad\qquad\qquad\qquad\qquad\qquad\displaystyle{+\partial_{x_j}(
 \sum_{k=m+2}^{+\infty}[\Delta_\nu, S_{k-3}a]\Delta_k (\partial_{x_h} u)).}
 \end{array}
 $$
In fact $\Delta_\nu$ and $\Delta_k$ commute so that
$$
\begin{array}{ll}
\displaystyle{\Delta_\nu(S_{m-1}aS_{m+1}w)-
S_{m-1}aS_{m+1}(\Delta_\nu w)}\\[0.3 cm]
\qquad\qquad\qquad\qquad\qquad\displaystyle{= \Delta_\nu(S_{m-1}aS_{m+1}w))-
S_{m-1}a\Delta_\nu (S_{m+1}w)}
\end{array}
$$
and similarly for the other term.

Let's consider 
$$
\partial_{x_j}([\Delta_\nu, S_{m-1}a]S_{m+1}(\partial_{x_h}u))=
\partial_{x_j}([\Delta_\nu, S_{m-1}a]\partial_{x_h}(S_{m+1}u).
$$
Looking at the support of the Fourier transform, it is easy to see that 
this term is identically equal to 0 if $\nu\geq m+4$. Moreover 
the support of the Fourier transform  is contained in $\{|\xi|\leq 2^{m+3}\}$.
From Bernstein's inequality we have
$$
\|\partial_{x_j}([\Delta_\nu, S_{m-1}a]\partial_{x_h}(S_{m+1}u))\|_{L^2}\leq 2^{m+3}
\|[\Delta_\nu, S_{m-1}a]\partial_{x_h}(S_{m+1}u)\|_{L^2}.
$$
On the other hand, using the result of \cite[Th. 35]{CMey} (see also \cite[Par. 3.6]{Tay}) we deduce that
$$
\|[\Delta_\nu, S_{m-1}a]\partial_{x_h}(S_{m+1}u)\|_{L^2}\leq  C
\|a\|_{Lip}\,\|S_{m+1}u\|_{L^2}\leq  C
\|a\|_{Lip}\,\|u\|_{L^2}.
$$
Consequently
$$
\|\partial_{x_j}([\Delta_\nu, S_{m-1}a]S_{m+1}(\partial_{x_h} u))\|_{L^2}\leq C 2^{m+3}
\|a\|_{Lip}\,\|u\|_{L^2},
$$
and, since $s\in (0,1)$, 
\begin{equation}
\begin{array}{ll}
\displaystyle{\sum_{\nu=0}^{+\infty}2^{-2\nu s}\|\partial_{x_j}([\Delta_\nu, 
S_{m-1}a]S_{m+1}(\partial_{x_h} u))\|_{L^2}^2}\\[0,3 cm]
\ \displaystyle{=
\sum_{\nu=0}^{m+3}2^{-2\nu s}\|\partial_{x_j}([\Delta_\nu, 
S_{ m-1}a]S_{m+1}(\partial_{x_h} u))\|_{L^2}^2
\leq C_{m, s} \|a\|_{Lip}^2\,\|u\|_{H^{1-s}}^2.}
\end{array}
\label{estcomm1}
\end{equation}

Let's consider 
$$
\partial_{x_j}\big (\sum_{k=m+2}^{+\infty}[\Delta_\nu, S_{k-3}a]\Delta_k (\partial_{x_h}u)\big)=\partial_{x_j}\big(\sum_{k=m+2}^{+\infty}[\Delta_\nu, S_{k-3}a]\partial_{x_h}(\Delta_k u)\big).
$$
Again looking at  the support of the Fourier transform, it is possible to see that
$ [\Delta_\nu, S_{k-3}a]\partial_{x_h}(\Delta_k u)$ is identically 0 if $|k-\nu|\geq 4$. 
Consequently the sum is on at most 7 terms: $\partial_{x_j}([\Delta_\nu, S_{\nu-6}a]\partial_{x_h}(\Delta_{\nu-3} u))+\dots +\partial_{x_j}([\Delta_\nu, S_{\nu}a]\\ \partial_{x_h}(\Delta_{\nu+3} u))$, each of them having the support of the Fourier transform contained in $\{|\xi|\leq C2^\nu\}$. 
Let's consider one of these terms, e.g. $\partial_{x_j}([\Delta_\nu, S_{\nu-3}a] \\ \partial_{x_h}(\Delta_{\nu} u))$, the computation for the others being similar. We have, from Bernstein's inequality
$$
\|\partial_{x_j}([\Delta_\nu, S_{\nu-3}a]\partial_{x_h}(\Delta_{\nu} u))\|_{L^2}
\leq C2^\nu \|[\Delta_\nu, S_{\nu-3}a]\partial_{x_h}(\Delta_{\nu} u)\|_{L^2}.
$$
and consequently, using again \cite[Th. 35]{CMey},
$$
\|\partial_{x_j}([\Delta_\nu, S_{\nu-3}a]\partial_{x_h}(\Delta_{\nu} u))\|_{L^2}
\leq C 2^\nu \|a\|_{Lip}\, \|\Delta_{\nu} u\|_{L^2}.
$$
Since $u\in H^{1-s}$ and consequently the sequence $\big(2^{\nu(1-s)}\|\Delta_\nu u\|_{L^2}\big)_\nu$ is in $l^2$ then the same is valid for
$\big(2^{-\nu s}\|\partial_{x_j}([\Delta_\nu, S_{\nu-3}a]\partial_{x_h}(\Delta_{\nu} u))\|_{L^2}\big)_\nu$ and
$$
\sum_{\nu=0}^{+\infty}2^{-2\nu s}\|\partial_{x_j}([\Delta_\nu, S_{\nu-3}a]\partial_{x_h}(\Delta_{\nu} u))\|_{L^2}^2\leq C_{m, s} \|a\|_{Lip}^2\, \|u\|_{H^{1-s}}^2.
$$
The computation of the other terms being similar we obtain
\begin{equation}
\sum_{\nu=0}^{+\infty}2^{-2\nu s}\|\partial_{x_j}(\sum_{k=m+2}^{+\infty}[\Delta_\nu, S_{k-3}a]\partial_{x_h}(\Delta_{k} u))\|_{L^2}^2\leq C_{m, s} \|a\|_{Lip}^2\, \|u\|_{H^{1-s}}^2.
\label{estcomm2}
\end{equation}
The estimate (\ref{estcomm}) follows from (\ref{estcomm1}) and (\ref{estcomm2}), concluding the proof.

\end{proof}

We end this subsection with a result on the adjoint of $T^m_a$ (see \cite[Prop. 3.8 and Prop. 3.11]{CMet}.
\begin{prop} \label{prop adj}
Let $m\in {\mathbb N}$, $a\in Lip$ and $u\in H^s$.
Then 
\begin{equation}
\label{adj}
\|(T^m_a-(T^m_a)^*)\partial_{x_j}u \|_{L^2}\leq C_m\|a\|_{Lip}\|u\|_{L^2}.
\end{equation}
\end{prop}

\begin{proof}
We  remark that
$$
(T^m_a-(T^m_a)^*)\partial_{x_j}u
=[S_{m-1}a, S_{m+1}] \partial_{x_j}u +\sum_{k=m+2}^{+\infty}[S_{k-3}a, \Delta_k] \partial_{x_j}u.
$$
From \cite[Th. 35]{CMey} we deduce that
$$
\|[S_{m-1}a, S_{m+1}] \partial_{x_j}u\| \leq C\|\nabla_x(S_{m-1} a)\|_{L^\infty}\|u\|_{L^2},
$$
and hence, from Prop \ref{prop car lip}, we obtain
\begin{equation}
\label{est adj 1}
\|[S_{m-1}a, S_{m+1}] \partial_{x_j}u\| \leq C\| a\|_{Lip}\|u\|_{L^2}.
\end{equation}
On the other hand we have that the support of Fourier transform of $[S_{k-3}a, \\\Delta_k] \partial_{x_j}u$ is contained in $\{2^{k-2}\leq|\xi|\leq 2^{k+2}\}$. Moreover it is easy to see that
$$
[S_{k-3}a, \Delta_k] \partial_{x_j}u=[S_{k-3}a, \Delta_k] \partial_{x_j}((\Delta_{k-1}+\Delta_{k}+\Delta_{k+1})u).
$$
Again from \cite[Th. 35]{CMey} and Proposition \ref{prop car lip} we have
$$
\begin{array}{ll}
\displaystyle{\|[S_{k-3}a, \Delta_k] \partial_{x_j}u\|_{L^2}}&=\displaystyle{
{\|[S_{k-3}a, \Delta_k] \partial_{x_j}((\Delta_{k-1}+\Delta_{k}+\Delta_{k+1})u)\|_{L^2}}}\\[0.3 cm]
&\leq \displaystyle{C \|a\|_{Lip}(\|\Delta_{k-1} u\|_{L^2}+\|\Delta_{k} u\|_{L^2}+\|\Delta_{k+1} u\|_{L^2})}.\\
\end{array}
$$
From Proposition \ref{prop car sob 2} we finally obtain that $\sum _{k=m+2}^{+\infty} 
[S_{k-3}a, \Delta_k] \partial_{x_j}u \in L^2$ and 
\begin{equation}\label{est adj 2}
\|\sum _{k=m+2}^{+\infty}  [S_{k-3}a, \Delta_k] \partial_{x_j}u\|\leq C_m \|a\|_{Lip}
\|u\|_{L^2}.
\end{equation}
The estimate (\ref{adj}) follows from (\ref{est adj 1}) and (\ref{est adj 2}).
\end{proof}


\subsection{Approximation and Carleman estimate}
We set $v(t,x)=e^{{1\over \gamma}\Phi(\gamma(T-t))}
u(t,x)$. The inequality (\ref{Carlest}) becomes:
for all $s\in (0,1)$, there  exist $\gamma_0$, $C>0$  such that
\begin{equation}
\begin{array}{ll}
\displaystyle{
\int_0^{T\over 2} \|\partial_t v+ \sum_{j,k = 1}^n \partial_{x_j} (a_{jk} (t,x)
\partial_{x_k} v)+ \Phi'(\gamma(T-t))v \|^2_{H^{-s}}\;
dt}\\[0.5 cm]
\qquad\qquad\qquad\displaystyle{\geq C\gamma^{1\over 2}\int_0^{T\over 2} 
(\|\nabla_{x} v\|^2_{H^{-s}}+\gamma^{1\over
2}\|v\|^2_{H^{-s}})\; dt,}
\end{array}
\label{Carlest1}
\end{equation}
for all $\gamma>\gamma_0$ and for all $v\in
C^\infty_0({\mathbb R}^{n+1})$
such that $\,{\rm supp}\, v\subseteq [0, T/2]\times {\mathbb R}^n$.
Using the Proposition \ref{prop pos} we fix the parameter $m$ in such a way that the
modified paraproduct associated to $(a_{jk})_{j,k}$ is a positive matrix operator.
From the second part of Proposition \ref{prop paraprod} (estimate (\ref{a-Ta})), the inequality (\ref{Carlest1}) will be deduced from the following
\begin{equation}
\begin{array}{ll}
\displaystyle{
\int_0^{T\over 2} \|\partial_t v+ \sum_{j,k = 1}^n \partial_{x_j} (T^m_{a_{jk}} 
\partial_{x_k} v)+ \Phi'(\gamma(T-t))v \|^2_{H^{-s}}\;
dt}\\[0.5 cm]
\qquad\qquad\qquad\displaystyle{\geq C\gamma^{1\over 2}\int_0^{T\over 2} 
(\|\nabla_{x} v\|^2_{H^{-s}}+\gamma^{1\over
2}\|v\|^2_{H^{-s}})\; dt,}
\end{array}
\label{Carlest2}
\end{equation}
as the quantity $2\|\sum_{j,k=1}^n\partial_{x_j} ((a_{jk}-T^m_{a_{jk}})\partial_{x_k}v)\|^2_{H^{-s}}$ can be absorbed by the right hand side part of (\ref{Carlest1}), possibly taking  different $C$ and $\gamma_0$.

Let's go back to the Littlewood-Paley decomposition; a consequence of Proposition
 \ref{prop car sob} is that, denoting from now on $\Delta_k u$ by $u_k$,  there exists $K_s>0$ such that
$$
 {1\over K_s}\sum_\nu2^{-2\nu s}\|u_\nu\|_{L^2}^2\leq \|u\|^2_{H^{-s}}\leq
K_s\sum_\nu 2^{-2\nu s}\|u_\nu\|^2_{L^2}
$$
 for all $u\in H^{-s}$. We have
$$
\begin{array}{l}
\displaystyle{\int_0^{T\over 2} \|\partial_t v+ \sum_{j,k = 1}^n \partial_{x_j} (T^m_{a_{jk}} 
\partial_{x_k} v)+ \Phi'(\gamma(T-t))v \|^2_{H^{-s}}
dt}\\[0.3 cm]
\displaystyle{\geq {1\over K_s}\int_0^{T\over 2}
\sum_\nu2^{-2\nu s}\|\Delta_\nu(\partial_t v+ \sum_{j,k = 1}^n \partial_{x_j} (T^m_{a_{jk}} 
\partial_{x_k} v)+ \Phi'(\gamma(T-t))v)\|_{L^2}^2\;
dt}\\[0.3 cm]
\displaystyle{\geq {1\over K_s}\int_0^{T\over 2}
\sum_\nu2^{-2\nu s}\|\partial_t v_\nu+ \sum_{j,k = 1}^n \partial_{x_j} (T^m_{a_{jk}} 
\partial_{x_k} v_\nu)+ \Phi'(\gamma(T-t))v_\nu}\\[0.3 cm]
\qquad\qquad\qquad\qquad\qquad\qquad\qquad\qquad\qquad\displaystyle{+\sum_{j,k = 1}^n
\partial_{x_j}([\Delta_\nu, T^m_{a_{jk}}]\partial_{x_k} v) \|_{L^2}^2\;
dt}\\[0.3 cm]
\displaystyle{\geq {1\over 2K_s}\int_0^{T\over 2}
\sum_\nu2^{-2\nu s}\|\partial_t v_\nu+ \sum_{j,k = 1}^n \partial_{x_j} (T^m_{a_{jk}} 
\partial_{x_k} v_\nu)+ \Phi'(\gamma(T-t))v_\nu\|_{L^2}^2\;dt}\\[0.5 cm]
\displaystyle{\quad\qquad\quad\qquad\quad\qquad\quad-{1\over K_s}\int_0^{T\over
2}\sum_\nu 2^{-2\nu s}\|
\sum_{j,k = 1}^n
\partial_{x_j}([\Delta_\nu, T^m_{a_{jk}}]\partial_{x_k} v) \|_{L^2}^2\; dt.}\\
\end{array}
$$
From the result of Proposition \ref{prop comm} is then immediate that  (\ref{Carlest2})
will be deduced from the same estimate from below for 
$$
\int_0^{T\over 2}
\sum_\nu2^{-2\nu s}\|\partial_t v_\nu+ \sum_{j,k = 1}^n \partial_{x_j} (T^m_{a_{jk}} 
\partial_{x_k} v_\nu)+ \Phi'(\gamma(T-t))v_\nu\|_{L^2}^2\;dt,
$$
again with  possibly different $C$ and $\gamma_0$.
We have
$$
\begin{array}{l} 
\displaystyle{\int_0^{T\over 2}\sum_\nu 2^{-2\nu s}\|\partial_tv_\nu+\sum_{jk} 
\partial_{x_j}(T^m_{a_{jk}}\partial_{x_k} v_\nu) + \Phi'(\gamma(T-t)) 
v_\nu\|_{L^2}^2\;dt}\\[0.3 cm] 
\displaystyle{\quad=\int_0^{T\over 
2}\sum_\nu 2^{-2\nu s}(\|\partial_tv_\nu\|_{L^2}^2+\|\sum_{jk} 
\partial_{x_j}(T^m_{a_{jk}}\partial_{x_k} v_\nu) + \Phi'(\gamma(T-t)) 
v_\nu\|_{L^2}^2}\\[0.5 cm] 
\displaystyle{\qquad+\gamma\Phi''(\gamma(T-t))\|v_\nu\|_{L^2}^2+2\,{\rm 
Re}\,\langle 
\partial_tv_\nu,\; \sum_{jk} \partial_{x_j}(T^m_{a_{jk}}\partial_{x_k} 
v_\nu)\rangle_{L^2})\;dt.}\\ 
\end{array} 
$$
We approximate the last term in the above equality using a well known technique 
which goes back to \cite {CDGS}. Let $\rho\in 
C^\infty_0({\mathbb R})$ with $\,{\rm supp}\, \rho\subseteq[-1/2,\; 1/2]$, 
$\int_{\mathbb R} \rho(s)\;ds=1$ and 
$\rho(s)\geq 0$ for all $s\in {\mathbb R}$; we set 
$$ 
a_{jk,\,\varepsilon}(t,x)=\int_{\mathbb R} a_{jk}(s,x){1\over 
\varepsilon}\rho({t-s\over \varepsilon})\; ds 
$$ 
for $\varepsilon\in\;]0, 1/2]$. We obtain from (\ref{reg}) that there exist $C$ such that 
\begin{equation}
|a_{jk,\,\varepsilon}(t,x)-a_{jk}(t,x)|\leq C\mu(\varepsilon)
\label{a-a_eps} 
\end{equation} and 
\begin{equation} 
|\partial_t a_{jk,\,\varepsilon}(t,x)|\leq C\,{\mu(\varepsilon)\over 
\varepsilon} ,
\label{a'_eps} 
\end{equation} 
\vskip0.2 cm 
\noindent for all $j,k=1\dots,n$ and for all $(t,x)\in [0,T]\times{\mathbb 
R}^n_x$ . We have 
$$ 
\begin{array}{l} 
\displaystyle{\int_0^{T\over 2}2\,{\rm Re}\,\langle 
\partial_tv_\nu,\; \sum_{jk} \partial_{x_j}(T^m_{a_{jk}}\partial_{x_k} 
v_\nu)\rangle_{L^2}\;dt}\\[0.3 cm] 
\displaystyle{=-2\,{\rm Re}\,\int_0^{T\over 2}\sum_{jk} 
\langle\partial_{x_j}\partial_t v_\nu, 
\; T^m_{a_{jk}}\partial_{x_k} v_\nu\rangle_{L^2}\;dt}\\[0.3 cm] 
\displaystyle{=-2\,{\rm Re}\,\int_0^{T\over 2}\sum_{jk} 
\langle\partial_{x_j}\partial_t v_\nu, 
\; (T^m_{a_{jk}}-T^m_{a_{jk,\,\varepsilon}})\partial_{x_k} 
v_\nu\rangle_{L^2}\;dt}\\[0.3 cm] 
\displaystyle{\qquad\qquad\qquad -2\,{\rm Re}\,\int_0^{T\over 2}\sum_{jk} 
\langle\partial_{x_j}\partial_t v_\nu, \; T^m_{a_{jk,\,\varepsilon}}\partial_{x_k} 
v_\nu\rangle_{L^2}\;dt.}\\ 
\end{array} 
$$ 
We remark that $T^m_{a_{jk}}-T^m_{a_{jk,\,\varepsilon}}=T^m_{a_{jk}-a_{jk,\,\varepsilon}} $ and consequently, from (\ref{estT_a}) and (\ref{a-a_eps}),  we have that
$$
\|(T^m_{a_{jk}}-T^m_{a_{jk,\,\varepsilon}})\partial_{x_k} v_\nu\|_{L^2}= 
\|T^m_{a_{jk}-a_{jk,\,\varepsilon}}\partial_{x_k} v_\nu\|_{L^2}
\leq C \mu(\varepsilon) \|\partial_{x_k} v_\nu\|_{L^2}.
$$
Moreover 
 $\|\partial_{x_j}v_\nu\|_{L^2}\leq 
2^{\nu+1}\|v_\nu\|_{L^2}$ and 
$\|\partial_{x_j}\partial_tv_\nu\|_{L^2}\leq 
2^{\nu+1}\|\partial_tv_\nu\|_{L^2}$ for all $\nu\in 
{\mathbb N}$. Hence
$$
\begin{array}{l} 
\displaystyle{|2\,{\rm Re}\,\int_0^{T\over 2}\sum_{jk} 
\langle\partial_{x_j}\partial_t v_\nu, \; 
(T^m_{a_{jk}}-T^m_{a_{jk,\,\varepsilon}})\partial_{x_k} v_\nu\rangle_{L^2}\;dt|}\\[0.3 cm] 
\displaystyle{\qquad\leq 2 C \mu(\varepsilon)\int_0^{T\over 2}\sum_{jk} 
\|\partial_{x_j}\partial_t v_\nu\|_{L^2}\;\|\partial_{x_k}v_\nu\|_{L^2}\; 
dt}\\[0.3 cm] 
\displaystyle{\qquad\qquad\leq {C\over N}\int_0^{T\over 
2}\|\partial_tv_\nu\|_{L^2}^2\;dt Ê+CN\, 
2^{4(\nu+1)}\mu(\varepsilon)\int_0^{T\over 2}\|v_\nu\|_{L^2}^2\;dt} \\ 
\end{array} 
$$ for all $N>0$ (note that $\mu(\varepsilon)^2\leq \mu(\varepsilon)$). 
On the other hand
$\partial _t(T^m_{a_{jk,\,\varepsilon}}w)=T_{\partial_ta_{jk,\,\varepsilon}}w +T_{a_{jk,\,\varepsilon}}\partial_t w$,  then, using also the fact that the matrix $(a_{jk})_{j,k}$ is real and symmetric, 
$$
\begin{array}{ll}
\displaystyle{-2\,{\rm Re}\,\int_0^{T\over 2}\sum_{jk} 
\langle\partial_{x_j}\partial_t v_\nu, \; T^m_{a_{jk,\,\varepsilon}}\partial_{x_k} 
v_\nu\rangle_{L^2}\;dt}\\[0.5 cm]
=\displaystyle{\int_0^{T\over 2}\sum_{jk}(\langle\partial_{x_j} v_\nu,T^m_{\partial_t a_{jk,\,\varepsilon}}\partial_{x_k}v_\nu\rangle_{L^2} +\langle (T^m_{a_{jk,\,\varepsilon}}-(T^m_{a_{jk,\,\varepsilon}})^*) \partial_{x_j}v_\nu, \partial_{x_k}\partial_t v_\nu\rangle_{L^2})dt.}
\end{array}
$$
From (\ref{estT_a}) and (\ref{a'_eps}) we deduce
$$
|\int_0^{T\over 2}\sum_{jk} \langle\partial_{x_j} v_\nu,T^m_{\partial_t a_{jk,\,\varepsilon}}\partial_{x_k}v_\nu\rangle_{L^2}\,dt|
\leq  C\, 2^{2(\nu+1)}\, 
{\mu(\varepsilon)\over 
\varepsilon}\int_0^{T\over 2} 
\|v_\nu\|_{L^2}^2\; dt,
$$
and, from (\ref{adj}) and (\ref{a-a_eps}),
$$
\begin{array}{ll}
\displaystyle{|\int_0^{T\over 2}\sum_{jk}\langle (T^m_{a_{jk,\,\varepsilon}}-(T^m_{a_{jk,\,\varepsilon}})^*) \partial_{x_j}v_\nu, \partial_{x_k}\partial_t v_\nu\rangle_{L^2})dt|}\\[0.3 cm] 
\displaystyle{\qquad\leq 2 C \mu(\varepsilon)\int_0^{T\over 2}\sum_{jk} 
\|\partial_{x_j} v_\nu\|_{L^2}\;\|\partial_{x_k}\partial_tv_\nu\|_{L^2}\; 
dt}\\[0.5 cm]
\displaystyle{\qquad\qquad\leq {C\over N}\int_0^{T\over 
2}\|\partial_tv_\nu\|_{L^2}^2\;dt Ê+CN\, 
2^{4(\nu+1)}\mu(\varepsilon)\int_0^{T\over 2}\|v_\nu\|_{L^2}^2\;dt} 
\end{array} 
$$ for all $N>0$. Choosing suitably N, we finally obtain
\begin{equation} 
\begin{array}{l} 
\displaystyle{\int_0^{T\over 2}\|\partial_tv_\nu+\sum_{jk} 
\partial_{x_j}(T^m_{a_{jk}}\partial_{x_k} v_\nu) + \Phi'(\gamma(T-t)) 
v_\nu\|^2_{L^2}\;dt}\\[0.3 cm] 
\displaystyle{\geq\int_0^{T\over 2}(\|\sum_{jk} 
\partial_{x_j}(T^m_{a_{jk}}\partial_{x_k} v_\nu) + \Phi'(\gamma(T-t)) v_\nu\|^2_{L^2} 
+\gamma\Phi''(\gamma(T-t))\|v_\nu\|^2_{L^2}}\\[0.3 cm] 
\displaystyle{\qquad\qquad-C( 
2^{4(\nu+1)}\,\mu(\varepsilon)+ 
2^{2(\nu+1)}\, {\mu(\varepsilon)\over \varepsilon}) 
\|v_\nu\|^2_{L^2})\;dt.}\\ 
\end{array} 
\label{finalest} 
\end{equation}

\subsection{End of the proof of the Carleman estimate}

From now on the proof is exactly the same as in \cite[Par. 3.2]{DSP}. We detail it for the reader's convenience.
Let $\nu=0$. From (\ref{weightatinfinity}) we can choose $\gamma_0>0$ 
such that 
$\Phi''(\gamma(T-t))\geq 1$ for all $\gamma>\gamma_0$ and for all 
$t\in [0,\, T/2]$. Taking now $\varepsilon=1/2$ we obtain from (\ref{finalest}) that 
$$ 
\begin{array}{l} 
\displaystyle{\int_0^{T\over 2}\|\partial_tv_0+\sum_{jk} 
\partial_{x_j}(T^m_{a_{jk}}\partial_{x_k} v_0) + \Phi'(\gamma(T-t)) 
v_0\|^2_{L^2}\;dt}\\[0.3 cm] 
\displaystyle{\qquad\qquad\qquad\qquad\geq\int_0^{T\over 
2}(\gamma-16C\mu({1\over 2}))\|v_0\|^2_{L^2}\;dt}\\ 
\end{array} 
$$ for all $\gamma>\gamma_0$. Possibly choosing a larger $\gamma_0$ we 
have, again for all 
$\gamma>\gamma_0$, 
\begin{equation}
\int_0^{T\over 2}\|\partial_tv_0+\sum_{jk} 
\partial_{x_j}(T^m_{a_{jk}}\partial_{x_k} v_0) + \Phi'(\gamma(T-t)) 
v_0\|^2_{L^2}
\geq {\gamma\over 2} \int_0^{T\over 2}\|v_0\|^2_{L^2}\; dt.
\label {estnu=0} 
\end{equation}

Let now $\nu\geq 1$.  We take $\varepsilon=2^{-2\nu}$. ÊWe obtain from 
(\ref{finalest}) that 
$$
\begin{array}{l} 
\displaystyle{\int_0^{T\over 2}\|\partial_tv_\nu+\sum_{jk} 
\partial_{x_j}(T^m_{a_{jk}}\partial_{x_k} v_\nu) + \Phi'(\gamma(T-t)) 
v_\nu\|^2_{L^2}\;dt}\\[0.3 cm] 
\displaystyle{\qquad\geq\int_0^{T\over 2}(\|\sum_{jk} 
\partial_{x_j}(T^m_{a_{jk}}\partial_{x_k} v_\nu) + \Phi'(\gamma(T-t)) v_\nu\|^2_{L^2} 
}\\[0.3 cm] 
\displaystyle{\qquad\qquad\qquad+\gamma\Phi''(\gamma(T-t))\|v_\nu\|^2_{L^2}-K\, 
2^{4\nu}\,\mu(2^{-2\nu})\|v_\nu\|^2_{L^2})\; dt}\\[0.3 cm] 
\displaystyle{\qquad\geq\int_0^{T\over 2}((\|\sum_{jk} 
\partial_{x_j}(T^m_{a_{jk}}\partial_{x_k} v_\nu)\|_{L^2} - \Phi'(\gamma(T-t))\| 
v_\nu\|_{L^2})^2 }\\[0.3 cm] 
\displaystyle{\qquad\qquad\qquad+\gamma\Phi''(\gamma(T-t))\|v_\nu\|^2_{L^2}-K\, 
2^{4\nu}\,\mu(2^{-2\nu})\|v_\nu\|^2_{L^2})\; dt} 
\end{array} 
$$ where $K=16C$. ÊOn the other hand, from (\ref{pos}), recalling that in this case 
$\|\nabla v_\nu\|\geq  2^{\nu-1}\|v_\nu\|$,
 we have 
\begin{equation}
\label{posest} 
\begin{array}{l} 
\displaystyle{ \|\sum_{jk} \partial_{x_j}(T^m_{a_{jk}}\partial_{x_k} v_\nu)\|_{L^2} 
\;\|v_\nu\|_{L^2}\geq 
|\langle\sum_{jk} \partial_{x_j}(T^m_{a_{jk}}\partial_{x_k} v_\nu),\; 
v_\nu\rangle_{L^2}|}\\[0.3 cm] 
\displaystyle{\qquad\qquad\geq|\sum_{jk}\langle T^m_{a_{jk}}\partial_{x_k} v_\nu,\; 
\partial_{x_j}v_\nu\rangle_{L^2}| 
\geq {\lambda_0\over 2}\|\nabla v_\nu\|^2_{L^2}\geq {\lambda_0\over 8}\, 
2^{2\nu}\,\| v_\nu\|^2_{L^2}.}\\ 
\end{array} 
\end{equation}

Suppose first that $\Phi'(\gamma(T-t))\leq{\lambda_0\over 8}\, 
2^{2\nu}$. Then from 
(\ref{posest}) we deduce that 
$$ 
Ê\|\sum_{jk} \partial_{x_j}(T^m_{a_{jk}}\partial_{x_k} v_\nu)\|_{L^2} - 
\Phi'(\gamma(T-t))\| v_\nu\|_{L^2}\geq 
{\lambda_0\over 16}\, 2^{2\nu}\|v_\nu\|_{L^2} 
$$ and then, using also the fact that $\Phi''(\gamma(T-t))\geq 1$, we obtain 
that 
$$
\begin{array}{l} 
\displaystyle{\int_0^{T\over 2}((\|\sum_{jk} 
\partial_{x_j}(T^m_{a_{jk}}\partial_{x_k} v_\nu)\|_{L^2} - \Phi'(\gamma(T-t))\| 
v_\nu\|_{L^2})^2 }\\[0.3 cm] 
\displaystyle{\qquad\qquad\qquad+\gamma\Phi''(\gamma(T-t))\|v_\nu\|_{L^2}^2-K\, 
2^{4\nu}\,\mu(2^{-2\nu})\|v_\nu\|_{L^2}^2\; )dt}\\[0.3 cm] 
\displaystyle{\quad \geq \int_0^{T\over 2}(({\lambda_0\over16}\, 
2^{2\nu})^2+\gamma -K\, 
2^{4\nu}\,\mu(2^{-2\nu})\|v_\nu\|_{L^2}^2)\; dt}\\[0.5 cm] 
\displaystyle{\qquad\geq\int_0^{T\over 2}(({1\over 2}({\lambda_0\over 16})^2-K\mu(2^{-2\nu}))\, 
2^{4\nu}+{\gamma\over 3})\|v_\nu\|_{L^2}^2\; dt }\\[0.5 cm]
\displaystyle{\qquad\qquad\qquad\qquad+\int_0^{T\over 2}({1\over 2}({\lambda_0\over 16})^2\,2^{4\nu}
+{2\over 3}\gamma)\|v_\nu\|_{L^2}^2\; dt.}\\ 
\end{array} 
$$
Since $\lim_{\nu\to+\infty} \mu(2^{-2\nu})=0$, there exists $\gamma_0>0$ such that 
$$({1\over 2}({\lambda_0\over 16})^2-K\mu(2^{-2\nu}))\,2^{4\nu} + 
{\gamma\over 3}\geq 0$$ for all $\gamma\geq \gamma_0$ and for all $\nu\geq 1$.
Consequently there exist 
$\gamma_0$ and Ê$c>0$ not depending on $\nu$ such that
\begin{equation} 
\begin{array}{l} 
\displaystyle{\int_0^{T\over 2}((\|\sum_{jk} 
\partial_{x_j}(T^m_{a_{jk}}\partial_{x_k} v_\nu)\|_{L^2} - \Phi'(\gamma(T-t))\| 
v_\nu\|_{L^2})^2 }\\[0.5 cm] 
\displaystyle{\qquad\qquad\qquad+\gamma\Phi''(\gamma(T-t))\|v_\nu\|_{L^2}^2-K\, 
2^{4\nu}\,\mu(2^{-2\nu})\|v_\nu\|_{L^2}^2\; )dt}\\[0.5 cm] 
\displaystyle{\geq\int_0^{T\over 2}({1\over 2}({\lambda_0\over 16})^2\, 
2^{4\nu}+{2\over 3}\gamma)\|v_\nu\|_{L^2}^2\; dt \geq 
\int_0^{T\over 2}({\gamma\over 2}+c\gamma^{1\over 2}\, 
2^{2\nu})\|v_\nu\|_{L^2}^2\; dt}\\ 
\end{array} 
\label{estphismall} 
\end{equation} for all $\gamma\geq \gamma_0$.

If on the contrary 
$\Phi'(\gamma(T-t))\geq{\lambda_0\over 16}\, 2^{2\nu}$ then, using (\ref{diffeq}), 
the fact that 
$\lambda_0\leq 1$ and the properties of $\mu$, 
$$ 
\begin{array}{l} 
\displaystyle{\Phi''(\gamma(T-t))= (\Phi'(\gamma(T-t))^2\mu({1\over 
\Phi'(\gamma(T-t))})}\\[0.3 cm] 
\displaystyle{\qquad\qquad\qquad\geq ({\lambda_0\over 16})^2\, 2^{4\nu}\, 
\mu({16\over 
\lambda_0}\,2^{-2\nu}) 
\geq ({\lambda_0\over 16})^2\, 2^{4\nu}\, \mu(2^{-2\nu}).}\\ 
\end{array} 
$$ 
Hence also in this case there exist $\gamma_0$ and Ê$c>0$ such that 
\begin{equation} 
\begin{array}{l} 
\displaystyle{\int_0^{T\over 2}((\|\sum_{jk} 
\partial_{x_j}(T^m_{a_{jk}}\partial_{x_k} v_\nu)\|_{L^2} - \Phi'(\gamma(T-t))\| 
v_\nu\|_{L^2})^2 }\\[0.3 cm] 
\displaystyle{\qquad\qquad\qquad+\gamma\Phi''(\gamma(T-t))\|v_\nu\|_{L^2}^2-K\, 
2^{4\nu}\,\mu(2^{-2\nu})\|v_\nu\|_{L^2}^2)\; dt}\\[0.3 cm] 
\displaystyle{\quad \geq \int_0^{T\over 2}({\gamma\over 2}+({\gamma\over 
2}({\lambda_0\over 
16})^2-K)\, 2^{4\nu}\mu(2^{-2\nu}))\|v_\nu\|_{L^2}^2\; dt}\\[0.3cm] 
\displaystyle{\qquad\geq \int_0^{T\over 2}({\gamma\over 2}+c\gamma\, 
2^{2\nu})\|v_\nu\|_{L^2}^2\; dt}\\ 
\end{array} 
\label{estphibig} 
\end{equation} for all $\gamma\geq \gamma_0$ and for all $\nu\geq 1$. Putting together (\ref{estphismall}) 
and (\ref{estphibig}) we have 
that there exist $\gamma_0$ and $c>0$ such that 
\begin{equation} 
\begin{array}{l} 
\displaystyle{\int_0^{T\over 2}\|\partial_t v_\nu+ \sum_{jk} 
\partial_{x_j}(T^m_{a_{jk}}\partial_{x_k} v_\nu) + \Phi'(\gamma(T-t)) v_\nu\|_{L^2}^2 
\, dt}\\[0.3 cm] 
Ê\displaystyle{\qquad \geq\int_0^{T\over 2}({\gamma\over 2}+c\gamma^{1\over 
2}\, 
2^{2\nu})\|v_\nu\|_{L^2}^2\; dt}\\ 
\end{array} 
\label{estnu>0} 
\end{equation} for all $\nu\geq 1$ and for all $\gamma\geq \gamma_0$. 

Form (\ref{estnu=0}) and (\ref{estnu>0}) we get that there exist $\gamma_0$ and 
$ c>0$ such that 
\begin{equation} 
\begin{array}{l} 
\displaystyle{\int_0^{T\over 2}\sum_\nu 2^{-2\nu s}\|\partial_t v_\nu+ \sum_{jk} 
\partial_{x_j}(T^m_{a_{jk}}\partial_{x_k} v_\nu) + \Phi'(\gamma(T-t)) v_\nu\|_{L^2}^2 
\, dt}\\[0.3 cm] 
Ê\displaystyle{\qquad \geq c \gamma^{1\over 2}\int_0^{T\over 
2}\sum_\nu2^{-2\nu s}(\|\nabla v_\nu\|_{L^2}^2+\gamma^{1\over 
2}\|v_\nu\|_{L^2}^2  
)\; dt}\\ 
\end{array} 
\label{finalfinalest} 
\end{equation} for all $\gamma\geq \gamma_0$.



\subsection*{Acknowledgment}
The authors would like to thank prof. Marius Paicu for useful  
and interesting discussions on topics related to this paper.

\end{document}